\documentclass[12pt,twoside]{amsart}

\usepackage{amssymb}
\usepackage{graphicx}
\usepackage[pdfauthor={David Cook II, Uwe Nagel},
            pdftitle={Resolving signed lozenge tilings},
            pdfsubject={Enumerative combinatorics},
            pdfkeywords={determinants, lozenge tilings, non-intersecting lattice paths, perfect matchings},
            pdfproducer={LaTeX with hyperref},
            pdfcreator={latex->dvips->ps2pdf},
            pdfpagemode=UseOutlines,
            bookmarksopen=false,
            letterpaper,
            bookmarksnumbered=true,
            pdfborder={0 0 0},
            colorlinks=false]{hyperref}
\usepackage[margin=1in]{geometry}

\newcommand{\dntri}{\bigtriangledown}
\newcommand{\uptri}{\triangle}

\newcommand{\ZZ}{\mathbb{Z}}
\newcommand{\PS}{\mathfrak{S}}

\DeclareMathOperator{\per}{perm}
\DeclareMathOperator{\sgn}{sgn}
\DeclareMathOperator{\msgn}{msgn}
\DeclareMathOperator{\lpsgn}{lpsgn}

\def\urltilda{\kern -.15em\lower .7ex\hbox{\~{}}\kern .04em}

\numberwithin{figure}{section}
\numberwithin{equation}{section}

\newtheorem{theorem}{Theorem}[section]
\newtheorem{lemma}[theorem]{Lemma}
\newtheorem{proposition}[theorem]{Proposition}
\newtheorem{corollary}[theorem]{Corollary}

\theoremstyle{definition}
\newtheorem{definition}[theorem]{Definition}
\newtheorem{remark}[theorem]{Remark}
\newtheorem{example}[theorem]{Example}

\begin{document}

\title{Signed lozenge tilings}

\author[D.\ Cook II]{David Cook II${}^{\star}$}
\address{Department of Mathematics \& Computer Science, Eastern Illinois University, Charleston, IL 46616}
\email{\href{mailto:dwcook@eiu.edu}{dwcook@eiu.edu}}

\author[U.\ Nagel]{Uwe Nagel}
\address{Department of Mathematics, University of Kentucky, 715 Patterson Office Tower, Lexington, KY 40506-0027}
\email{\href{mailto:uwe.nagel@uky.edu}{uwe.nagel@uky.edu}}

\thanks{
    Part of the work for this paper was done while the authors were partially supported by the National Security Agency
    under Grant Number H98230-09-1-0032.
    The second author was also partially supported by the National Security Agency under Grant Number H98230-12-1-0247
    and by the Simons Foundation under grants \#208869 and \#317096.\\
    \indent ${}^{\star}$ Corresponding author.}


\begin{abstract}
    It is well-known that plane partitions, lozenge tilings of a hexagon, perfect matchings on a honeycomb graph, and families of
    non-intersecting lattice paths in a hexagon are all in bijection.  In this work we consider regions that are more general than hexagons. They are obtained by further removing upward-pointing triangles. We
    call the resulting shapes triangular regions.  We establish signed versions of the latter three bijections for  triangular regions.  We first investigate the tileability
    of triangular regions by lozenges.  Then we use perfect matchings and families of non-intersecting lattice paths to define two signs of a
    lozenge tiling.  Using a new method that we  call resolution of a puncture, we show that the two signs are in fact equivalent. As a consequence, we obtain the equality of determinants, up to sign, that enumerate signed perfect matchings and signed families of lattice paths of a triangular region, respectively. We also describe triangular regions, for which the signed enumerations agree with the unsigned enumerations.
\end{abstract}

\maketitle

\section{Introduction} \label{sec:intro}

It is a useful and well-known fact that plane partitions in an $a \times b \times c$ box, lozenge tilings of a hexagon with side lengths $(a,b,c)$, families of non-intersecting lattice path in such a hexagon, and perfect matchings of a suitable honeycomb graph are all in bijection. In this work we refine the latter three bijections by establishing signed versions of them for regions that are more general than hexagons.

More specifically, we consider certain subregions of a triangular region $\mathcal{T}_d$.  The latter is an equilateral triangle of side length $d$ subdivided by equilateral triangles of side length one.
We view a hexagon with side lengths $a, b, c$ as the region obtained by removing triangles of side lengths $a, b$, and $c$ at the vertices
of $\mathcal{T}_d$, where $d = a + b + c$. More generally, we consider subregions $T \subset \mathcal{T} = \mathcal{T}_d$ (for some $d$) that arise from $\mathcal{T}$ by removing upward-pointing
triangles, each of them being a union of unit triangles.  We refer to the removed upward-pointing triangles as \emph{punctures}.
The punctures  may
overlap (see Figure \ref{fig:triregion-intro}). We call the resulting subregions of $\mathcal{T}$ \emph{triangular subregions}.
Such a region is said to be \emph{balanced} if it contains as many upward-pointing unit triangles as down-pointing pointing
unit triangles. For example, hexagonal subregions are balanced. Lozenge tilings of triangular subregions have been studied
in several areas. For example, they are used in statistical mechanics for modeling bonds in dimers (see, e.g., \cite{Ke})
or in statistical mechanics when studying phase transitions (see, e.g., \cite{Ci-2005}).

\begin{figure}[!ht]
    \begin{minipage}[b]{0.48\linewidth}
        \centering
        \includegraphics[scale=1]{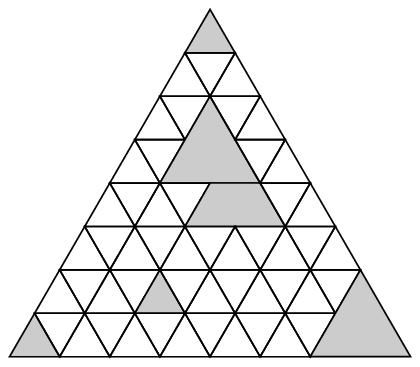}
    \end{minipage}
    \begin{minipage}[b]{0.48\linewidth}
        \centering
        \includegraphics[scale=1]{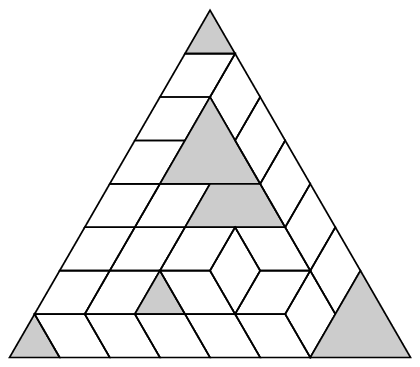}
    \end{minipage}
    \caption{A triangular region  together
        with one of its $13$ lozenge tilings.}
    \label{fig:triregion-intro}
\end{figure}

For an arbitrary triangular region, the bijection between lozenge tilings and plane partitions breaks down. However, there are still bijections between lozenge tilings, perfect matchings, and families of lattice paths. Here we establish a signed version of these bijections. In particular, we show that, for each balanced triangular
region $T$, there is a bijection between the signed perfect matchings and the signed families of non-intersecting
lattice paths. This is achieved via the links to lozenge tilings.

\begin{figure}[!ht]
   \begin{minipage}[b]{0.42\linewidth}
        \centering
        \includegraphics[scale=1]{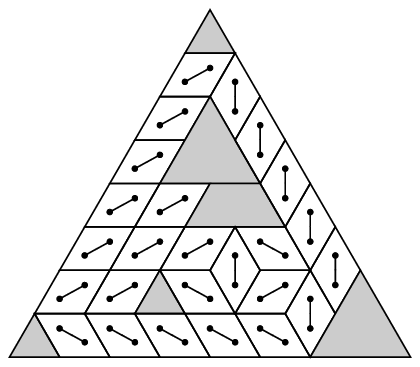}\\
        \emph{A perfect matching.}
    \end{minipage}
    \begin{minipage}[b]{0.48\linewidth}
        \centering
        \includegraphics[scale=1]{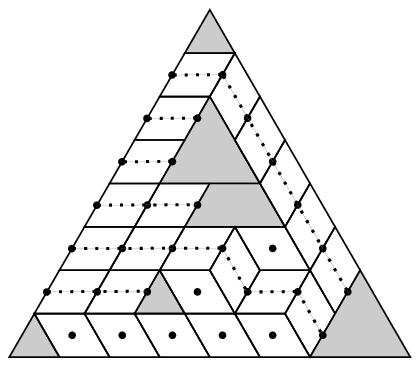}\\
        \emph{A family of non-intersecting lattice paths.}
    \end{minipage}
    \caption{Bijections to lozenge tilings.}
    \label{fig:bijections}
\end{figure}

Indeed, the perfect matchings determined by any triangular region $T$ can be enumerated by the permanent of a zero-one matrix $Z(T)$ that is the
bi-adjacency matrix of a bipartite graph. This suggests to introduce the sign of a perfect matching such that the
signed perfect matchings are enumerated by the determinant of $Z(T)$. We call this sign the \emph{perfect matching sign}
of the lozenge tiling that corresponds to the perfect matching (see Definition \ref{def:pm-sign}).

Using the theory pioneered by Gessel and Viennot \cite{GV-85}, Lindstr\"om \cite{Li},  Stembridge \cite{Stembridge},
and Krattenthaler \cite{Kr-95}, the sets of signed families of non-intersecting lattice paths in $T$
can be enumerated by the determinant of a matrix $N(T)$ whose entries are binomial coefficients. We define the sign used in this enumeration as the \emph{lattice path sign} of the corresponding lozenge tiling of the region $T$ (see Definition \ref{def:nilp-sign}).

Typically, the matrix $N(T)$ is much smaller than the
matrix $Z(T)$. However, the entries of $N(T)$ can be much bigger than one.

In order to compare enumerations of signed perfect matchings and signed lattice paths  we introduce a new combinatorial construction that we call \emph{resolution of a puncture}.
Roughly speaking, it replaces a triangular subregion with a fixed lozenge tiling by a larger triangular subregion with a
compatible lozenge tiling and one puncture less. Carefully analyzing the change of sign under resolutions of punctures and
using induction on the number of punctures of a given region, we establish that, for each balanced triangular subregion,
the perfect matching sign and the lattice path sign  are in fact equivalent, and thus (see Theorem \ref{thm:detZN})
\[
    |\det Z(T)| = |\det N(T)|.
\]
The proof also reveals instances where the absolute value of $\det Z(T)$ is equal to the permanent of $Z(T)$.
This includes hexagonal regions, for which the result is well-known.

The results of this paper will be used in forthcoming work \cite{CN-small-type} in order to study the so-called Weak Lefschetz Property \cite{HMNW} of monomial ideals. The latter is an algebraic property that has important connections to combinatorics. For example, it has been used for establishing unimodality results and the g-Theorem on the face vectors of simplicial polytopes (see, e.g., \cite{Stanley-1980, St-faces}).

The paper is organized as follows.  In Section~\ref{sec:trireg}, we introduce triangular regions and  establish a criterion for the tileability of such a region.
In Section~\ref{sec:signed}, we introduce the perfect matching and lattice path signs for a lozenge tiling.
Section~\ref{sec:resolution} contains our main results. There we introduce the method of resolving a puncture and use it to  prove the equivalence of the two signs.

\section{Tiling triangular regions with lozenges} \label{sec:trireg}

In this section, we introduce a generalization of hexagonal regions, which we call triangular regions, and we investigate
the tileability of such regions. We use monomial ideals as a bookkeeping device.

\subsection{Triangular regions and monomial ideals}\label{sub:ideal}~

Let $I$ be a monomial ideal of  a standard graded polynomial ring $R= K[x,y,z]$ over a field $K$. Thus,  $I$ has  a unique generating set of monomials with least cardinality. Its elements are called the minimal generators of $I$.
We denote the degree $d$ component of the graded ring $R/I$ by $[R/I]_d$. Note that the degree $d$ monomials  of $R$
that are \emph{not} in $I$ form a $K$-basis of $[R/I]_d$.

Let $d \geq 1$ be an integer. Consider an equilateral triangle of side length $d$ that is composed of $\binom{d}{2}$
downward-pointing ($\dntri$) and $\binom{d+1}{2}$ upward-pointing ($\uptri$) equilateral unit triangles. We label the
downward- and upward-pointing unit triangles by the monomials in $[R]_{d-2}$ and $[R]_{d-1}$, respectively, as
follows: place $x^{d-1}$ at the top, $y^{d-1}$ at the bottom-left, and $z^{d-1}$ at the bottom-right, and continue
labeling such that, for each pair of an upward- and a downward-pointing triangle that share an edge, the label of the
upward-pointing triangle is obtained from the label of the downward-pointing triangle by multiplying with a variable.
The resulting labeled triangular region is the \emph{triangular region (of $R$) in degree $d$}
and is denoted $\mathcal{T}_d$. See Figure~\ref{fig:triregion-R}(i) for an illustration.

\begin{figure}[!ht]
    \begin{minipage}[b]{0.48\linewidth}
        \centering
        \includegraphics[scale=1]{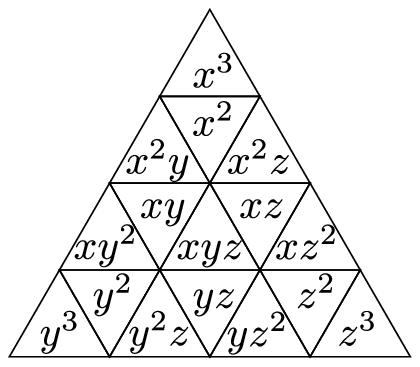}\\
        \emph{(i) $\mathcal{T}_4$}
    \end{minipage}
    \begin{minipage}[b]{0.48\linewidth}
        \centering
        \includegraphics[scale=1]{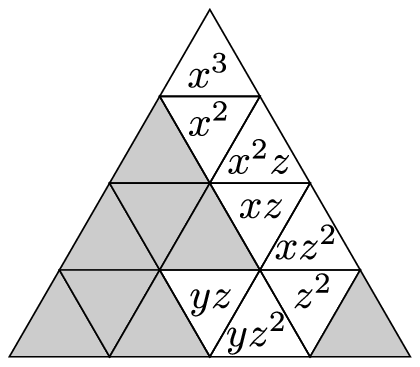}\\
        \emph{(ii) $T_4(xy, y^2, z^3)$}
    \end{minipage}
    \caption{A triangular region with respect to $R$ and with respect to $R/I$.}
    \label{fig:triregion-R}
\end{figure}

Throughout this manuscript we order the monomials of $R$ by using the \emph{graded reverse-lexico\-graphic order}, that is,
$x^a y^b z^c > x^p y^q z^r$ if either $a+b+c > p+q+r$ or $a+b+c = p+q+r$ and the \emph{last} non-zero entry in
$(a-p, b-q, c-r)$ is \emph{negative}. For example, in degree $3$,
\[
    x^3 > x^2y > xy^2 > y^3 > x^2z > xyz > y^2z > xz^2 > yz^2 > z^3.
\]
Thus in $\mathcal{T}_4$, see Figure~\ref{fig:triregion-R}(i), the upward-pointing triangles are ordered starting at
the top and moving down-left in lines parallel to the upper-left edge.

We generalise this construction to quotients by monomial ideals. Let $I$ be a monomial ideal of $R$. The
\emph{triangular region (of $R/I$) in degree $d$}, denoted by $T_d(I)$, is the part of $\mathcal{T}_d$ that is obtained
after removing the triangles labeled by monomials in $I$. Note that the labels of the downward- and
upward-pointing triangles in $T_d(I)$ form $K$-bases of $[R/I]_{d-2}$ and $[R/I]_{d-1}$, respectively. It is sometimes
more convenient to illustrate such regions with the removed triangles darkly shaded instead of being removed; both
illustration methods will be used throughout this manuscript. See Figure~\ref{fig:triregion-R}(ii) for an example.

Notice that the regions missing from $\mathcal{T}_d$ in $T_d(I)$ can be viewed as a union of (possibly overlapping)
upward-pointing triangles of various side lengths that include the upward- and downward-pointing triangles inside them.
Each of these upward-pointing triangles corresponds to a minimal generator of $I$ that has, necessarily, degree at most
$d-1$. We can alternatively construct $T_d(I)$ from $\mathcal{T}_d$ by removing, for each minimal generator $x^a y^b
z^c$ of $I$ of degree at most $d-1$, the \emph{puncture associated to $x^a y^b z^c$} which is an upward-pointing
equilateral triangle of side length $d-(a+b+c)$ located $a$ triangles from the bottom, $b$ triangles from the
upper-right edge, and $c$ triangles from the upper-left edge. See Figure~\ref{fig:triregion-punctures} for an example.
We call $d-(a+b+c)$ the \emph{side length of the puncture associated to $x^a y^b z^c$}, regardless of possible overlaps
with other punctures in $T_d (I)$.

\begin{figure}[!ht]
    \begin{minipage}[b]{0.48\linewidth}
        \centering
        \includegraphics[scale=1]{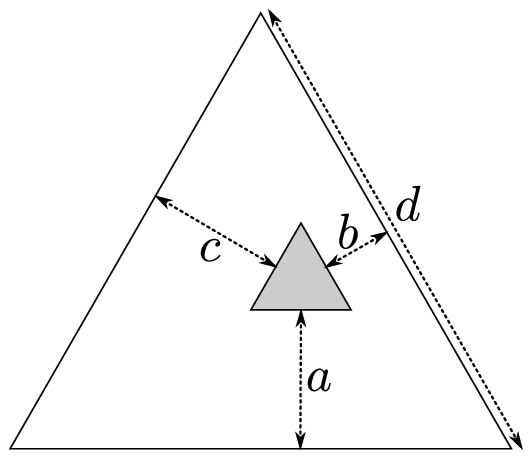}\\
        \emph{(i) $T_{d}(x^a y^b z^c)$}
    \end{minipage}
    \begin{minipage}[b]{0.48\linewidth}
        \centering
        \includegraphics[scale=1]{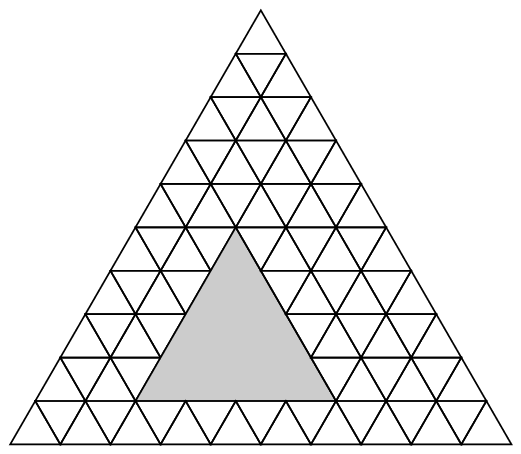}\\
        \emph{(ii) $T_{10}(xy^3z^2)$}
    \end{minipage}
    \caption{$T_d(I)$ as constructed by removing punctures.}
    \label{fig:triregion-punctures}
\end{figure}

We say that two punctures \emph{overlap} if they share at least an edge. Two punctures are said to be \emph{touching}
if they share precisely a vertex.

\subsection{Tilings with lozenges} \label{sub:tiling}

A \emph{lozenge} is a union of two unit equilateral triangles glued together along a shared edge, i.e., a rhombus with
unit side lengths and angles of $60^{\circ}$ and $120^{\circ}$. Lozenges are also called calissons and diamonds in the
literature.

Fix a positive integer $d$ and consider the triangular region $\mathcal{T}_d$ as a union of unit triangles. Thus a \emph{subregion}
$T \subset \mathcal{T}_d$ is a subset of such triangles. We retain their labels. We say that
a subregion $T$ is \emph{$\dntri$-heavy}, \emph{$\uptri$-heavy}, or \emph{balanced} if there are more downward pointing
than upward pointing triangles or less, or if their numbers are the same, respectively. A subregion is \emph{tileable}
if either it is empty or there exists a tiling of the region by lozenges such that every triangle is part of exactly one
lozenge. A tileable subregion is necessarily balanced as every unit triangle is part of exactly one lozenge.

Let $T \subset \mathcal{T}_d$ be any subregion. Given a monomial $x^a y^b z^c$ with degree less than $d$, the
\emph{monomial subregion} of $T$ associated to $x^a y^b z^c$ is the part of $T$ contained in the triangle $a$ units from
the bottom edge, $b$ units from the upper-right edge, and $c$ units from the upper-left edge. In other words, this
monomial subregion consists of the triangles that are in $T$ and the puncture associated to the monomial $x^a y^b z^c$.
See Figure~\ref{fig:triregion-subregion} for an example.

\begin{figure}[!ht]
    \includegraphics[scale=1]{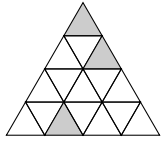}
    \caption{The monomial subregion of $T_{8}(x^7, y^7, z^6, x y^4 z^2, x^3 y z^2, x^4 y z)$
        (see Figure~\ref{fig:triregion-intro}) associated to $x y^2 z$.}
    \label{fig:triregion-subregion}
\end{figure}

Replacing a tileable monomial subregion by a puncture of the same size does not alter tileability.

\begin{lemma} \label{lem:replace-tileable}
    Let $T \subset \mathcal{T}_d$ be any subregion.  If the monomial subregion $U$ of $T$ associated to $x^a y^b z^c$ is tileable,
    then $T$ is tileable if and only if $T \setminus U$ is tileable.

    Moreover, each tiling of $T$ is obtained by combining a tiling of $T \setminus U$ and a tiling of $U$.
\end{lemma}
\begin{proof}
    Suppose $T$ is tileable, and let $\tau$ be a tiling of $T$.  If a tile in $\tau$ contains a downward-pointing triangle of $U$, then the
    upward-pointing triangle of this tile also is in $U$. Hence, if any lozenge in $\tau$ contains exactly one triangle of $U$, then
    it must be an upward-pointing triangle. Since $U$ is balanced, this would leave $U$ with a downward-pointing triangle that is
    not part of any tile, a contradiction. It follows that $\tau$ induces a tiling of $U$, and thus $T \setminus U$ is tileable.

    Conversely, if $T \setminus U$ is tileable, then a tiling of $T \setminus U$ and a tiling of $U$ combine to a tiling of $T$.
\end{proof}

Let $U \subset \mathcal{T}_d$ be a monomial subregion, and let $T, T' \subset \mathcal{T}_d$ be any subregions such that
$T \setminus U = T' \setminus U$. If $T \cap U$ and $T' \cap U$ are both tileable, then $T$ is tileable if and only if
$T'$ is, by Lemma \ref{lem:replace-tileable}. In other words, replacing a tileable monomial subregion of a triangular
region by a tileable monomial subregion of the same size does not affect tileability.

Using this observation, we find a tileability criterion of triangular regions
associated to monomial ideals. If it is satisfied the argument below constructs a tiling.

\begin{theorem} \label{thm:tileable}
    Let $T = T_d(I)$ be a balanced triangular region, where $I \subset R$ is any monomial ideal.  Then $T$ is tileable if and only if
    $T$ has no $\dntri$-heavy monomial subregions.
\end{theorem}
\begin{proof}
    Suppose $T$ contains a $\dntri$-heavy monomial subregion $U$.  That is, $U$ has more downward-pointing triangles than upward-pointing
    triangles.  Since the only triangles of $T \setminus U$ that share an edge with  $U$ are downward-pointing triangles, it is impossible to cover every
    downward-pointing triangle of $U$ with a lozenge.  Thus, $T$ is non-tileable.

    Conversely,  suppose $T$ has no $\dntri$-heavy monomial subregions.  In order to show that $T$ is tileable, we may also assume
    that $T$ has no non-trivial tileable monomial subregions by Lemma~\ref{lem:replace-tileable}.

    Consider any pair of touching or overlapping punctures in $\mathcal{T}_d$. The smallest monomial subregion $U$ containing both punctures
    is tileable. (In fact, such a monomial region is uniquely tileable by lozenges.)
    If further triangles stemming from other punctures of $T$ have been removed from $U$, then the resulting region
    $T \cap U$ becomes $\dntri$-heavy or empty. Thus, our assumptions imply that $T$ has no overlapping and no touching
    punctures.

    Now we proceed by induction on $d$. If $d \leq 2$, then $T$ is empty or consists of one lozenge.  Thus, it is tileable.
    Let $d \geq 3$, and let $U$ be the monomial subregion of $T$ associated to $x$, i.e., $U$ consists of  the upper $d-1$
    rows of $T$.  Let $L$ be the bottom row of $T$.  If $L$ does not contain part of a puncture of $T$, then $L$ is
    $\uptri$-heavy forcing $U$ to be a $\dntri$-heavy monomial subregion, contradicting an assumption on $T$.  Hence, $L$
    must contain part of at least one puncture of $T$.  See Figure~\ref{fig:thm-tileable}(i).

    \begin{figure}[!ht]
        \begin{minipage}[b]{0.48\linewidth}
            \centering
            \includegraphics[scale=1]{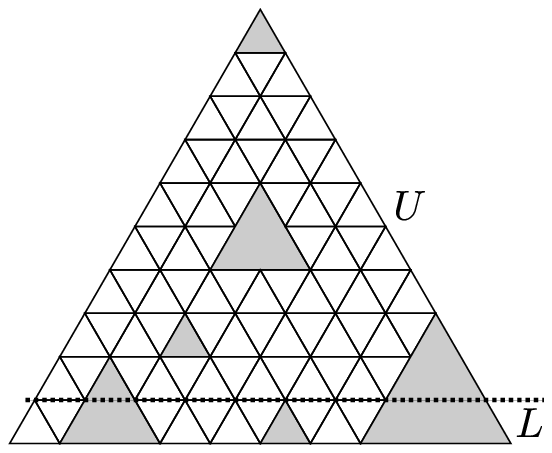}\\
            \emph{(i) The region $T$ split in to $U$ and $L$.}
        \end{minipage}
        \begin{minipage}[b]{0.48\linewidth}
            \centering
            \includegraphics[scale=1]{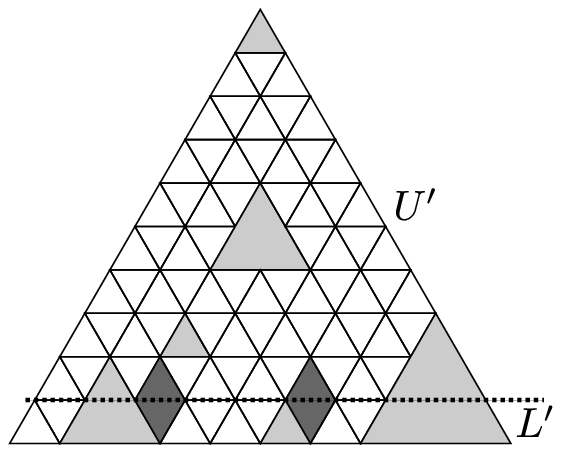}\\
            \emph{(ii) Creating $U'$ and $L'$.}
        \end{minipage}
        \caption{Illustrations for the proof of Theorem~\ref{thm:tileable}.}
        \label{fig:thm-tileable}
    \end{figure}

    Place an up-down lozenge in $T$ just to the right of each puncture along the bottom row \emph{except} the farthest right puncture.
    Notice that putting in all these tiles is  possible since punctures are non-overlapping and non-touching.  Let $U' \subset U$
    and $L' \subset L$ be the subregions that are obtained by removing the relevant upward-pointing and  downward-pointing triangles
    of the added lozenges from $U$ and $L$, respectively.  See Figure~\ref{fig:thm-tileable}(ii). Notice, $L'$ is uniquely tileable.

    As $T$ and $L'$ are balanced,  so is $U'$.   Assume $U'$ contains a monomial subregion $V'$ that is $\dntri$-heavy.
    Then $V' \neq U'$, and hence $V'$ fits into a triangle of side length $d-2$. Furthermore, the assumption on $T$ implies
    that $V'$ is not a monomial subregion of $U$. In particular, $V'$ must be located at the bottom of $U'$. Let $\tilde{V}$
    be the smallest monomial subregion of $U$ that contains $V'$. It is obtained from $V'$ by adding suitable upward-pointing
    triangles that are parts of the added lozenges.  Expand $\tilde{V}$ down one row to a monomial subregion $V$ of $T$.
    Thus, $V$ fits into a triangle of side length $d-1$ and is not $\dntri$-heavy. If $V$ is balanced, then, by induction,
    $V$ is tileable.  However, we assumed $T$ contains no such non-trivial regions.  Hence, $V$ is $\uptri$-heavy.
    Observe now that the region $V \cap L'$ is either balanced or has exactly one more upward-pointing triangle than
    downward-pointing triangles.  Since $V'$ is obtained from $V$ by removing $V \cap L$ and some of the added lozenges,
    it follows that $V'$ cannot be $\dntri$-heavy, a contradiction.

    Therefore, we have shown that each monomial subregion of $U'$ is not $\dntri$-heavy. By induction on $d$, we conclude
    that $U'$ is tileable. Using the lozenges already placed, along with the tiling of $L'$, we obtain a tiling of $T$.
\end{proof}

\begin{remark} \label{rem:complexity}
    The preceding proof yields a recursive construction of a canonical tiling of the triangular region.  In fact, the tiling can
    be seen as minimal, in the sense of Subsection~\ref{sub:nilp}.  Moreover, the theorem yields an exponential (in the number of
    punctures) algorithm to determine the tileability of a region.

    Thurston~\cite{Th} gave a linear (in the number of triangles) algorithm to determine the tileability of a \emph{simply-connected region},
    i.e., a region with a polygonal boundary.  Thurston's algorithm also yields a minimal canonical tiling.
\end{remark}

\section{Signed lozenge tilings} \label{sec:signed}

In Theorem~\ref{thm:tileable}, we established a tileability criterion for a triangular region.   Now we want to
\emph{enumerate} the lozenge tilings of a tileable triangular region $T_d(I)$. In fact, we introduce two ways for assigning a sign to a
lozenge tiling here and then compare the resulting enumerations in the next section.

In order to derive the (unsigned) enumeration, we consider the enumeration of perfect matchings of an associated bipartite
graph.  The permanent of its bi-adjacency matrix, a zero-one matrix,  yields
the desired enumeration.  We define a first sign of a lozenge tiling in such a way that the determinant of the bi-adjacency matrix gives a \emph{signed} enumeration of the perfect matchings of the graph and hence of lozenge tilings of $T_d(I)$.

We also introduce a second sign of a lozenge tiling by considering an enumeration of families of non-intersecting
lattice paths on an associated finite sub-lattice inside $T_d(I)$. This is motivated by the Lindstr\"om-Gessel-Viennot theory \cite{Li}, \cite{GV}. Using the sub-lattice, we generate a matrix whose entries are  binomial coefficients and whose determinant gives a signed enumeration
of families of non-intersecting lattice paths inside $T_d(I)$, hence of lozenge tilings.  The two signed enumerations appear to be different, but we show
that they are indeed the same, up to sign, in the following section.

~\subsection{Perfect matchings}\label{sub:pm}~\par

A subregion $T (G) \subset \mathcal{T}_d$ can be associated to a bipartite planar graph $G$ that is an induced subgraph
of the honeycomb graph. Lozenge tilings of $T(G)$ can be then associated to perfect matchings on $G$. The connection was
used by Kuperberg in~\cite{Kup}, the earliest citation known to the authors, to study symmetries on plane partitions.
Note that $T(G)$ is often called the \emph{dual graph} of $G$ in the literature (e.g., \cite{Ci-1997}, \cite{Ci-2005},
and \cite{Ei}). Here we begin with a subregion $T$  and then construct a suitable graph $G$.

Let $T \subset \mathcal{T}_d$ be any subregion. As above, we consider $T$ as a union of unit triangles. We associate to
$T$ a bipartite graph. First, place a vertex at the center of each triangle. Let $B$ be the set of centers of the
downward-pointing triangles, and let $W$ be the set of centers of the upward-pointing triangles. Consider both sets
ordered by the reverse-lexicographic ordering applied to the monomial labels of the corresponding triangles (see
Section~\ref{sub:ideal}). The \emph{bipartite graph associated to $T$}
is the bipartite graph $G(T)$ on the vertex set $B \cup W$ that has an edge between vertices $B_i \in B$ and $W_j \in W$
if the corresponding upward- and downward-pointing triangle share are edge. In other words, edges of $G(T)$ connect vertices
of adjacent triangles. See Figure~\ref{fig:build-pm}(i).

\begin{figure}[!ht]
    \begin{minipage}[b]{0.32\linewidth}
        \centering
        \includegraphics[scale=1]{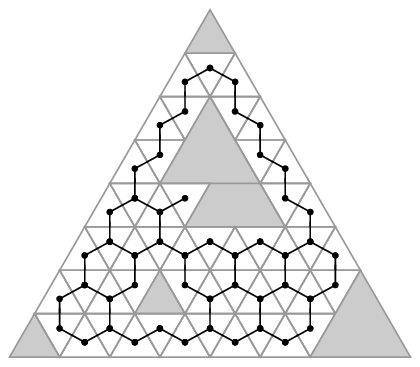}\\
        \emph{(i) The graph $G(T)$.}
    \end{minipage}
    \begin{minipage}[b]{0.32\linewidth}
        \centering
        \includegraphics[scale=1]{figs/build-pm-2}\\
        \emph{(ii) Selected covered edges.}
    \end{minipage}
    \begin{minipage}[b]{0.32\linewidth}
        \centering
        \includegraphics[scale=1]{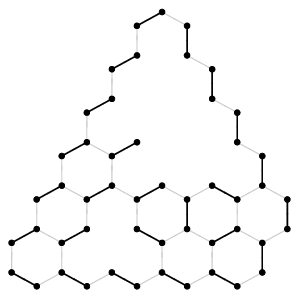}\\
        \emph{(iii) The perfect matching.}
    \end{minipage}
    \caption{The perfect
        matching  of the bipartite graph $G(T)$ associated to the tiling of $T$ in Figure~\ref{fig:triregion-intro}.}
    \label{fig:build-pm}
\end{figure}

Using the above ordering of the vertices, we define the \emph{bi-adjacency matrix}
of $T$ as the bi-adjacency matrix $Z(T) := Z(G(T))$ of the graph $G(T)$. It is the zero-one matrix $Z(T)$ of
size $\# B \times \# W$ with entries $Z(T)_{(i,j)}$ defined by
\begin{equation*}
    Z(T)_{(i,j)} =
    \begin{cases}
        1 & \text{if $(B_i, W_j)$ is an edge of $G(T)$ } \\
        0 & \text{otherwise.}
    \end{cases}
\end{equation*}

\begin{remark} \label{rem:Z-non-square}
    Note that $Z(T)$ is a square matrix if and only if the region $T$ is balanced.  Observe also that the construction
    of $G(T)$ and $Z(T)$ do not require any restrictions on $T$.  In particular, $T$ need not be
    balanced,  and so $Z(T)$ need not be square.
\end{remark}

A \emph{perfect matching of a graph $G$}
is a set of pairwise non-adjacent edges of $G$ such that each vertex is matched. There is a well-known bijection between lozenge
tilings of a balanced subregion $T$ and perfect matchings of $G(T)$. A lozenge tiling $\tau$ is transformed in to a perfect
matching $\pi$ by overlaying the triangular region $T$ on the bipartite graph $G(T)$ and selecting the edges of the graph that
the lozenges of $\tau$ cover. See Figures~\ref{fig:build-pm}(ii) and~(iii) for the overlayed image and the perfect matching by
itself, respectively.

\begin{remark}
    The graph $G(T)$ is a ``honeycomb graph,'' a type of graph that has been studied, especially for its perfect matchings.
    \begin{enumerate}
        \item In particular, honeycomb graphs are investigated for their connections to physics.  Honeycomb graphs model the bonds
            in dimers (polymers with only two structural units), and perfect matchings correspond to so-called \emph{dimer coverings}.
            Kenyon~\cite{Ke} gave a modern recount of explorations on dimer models, including random dimer coverings and their
            limiting shapes.  See the recent memoir~\cite{Ci-2005} of Ciucu for further results in this direction.
        \item Kasteleyn~\cite{Ka} provided, in 1967, a general method for computing the number of perfect matchings of a planar graph by means
            of a determinant.  In the following observation, we compute the number of perfect matchings on $G(T)$ by means of a permanent.
    \end{enumerate}
\end{remark}

Recall that the \emph{permanent} of an $n \times n$ matrix $M = (M_{(i, j)})$ is given by
\[
    \per{M} := \sum_{\sigma \in \PS_n} \prod_{i=1}^{n} M_{(i, \sigma(i))}.
\]

\begin{proposition} \label{pro:per-enum}
    Let $T \subset \mathcal{T}_d$ be a non-empty balanced subregion.  Then the lozenge tilings of $T$ and the perfect matchings of $G(T)$ are both
    enumerated by $\per{Z(T)}$.
\end{proposition}
\begin{proof}
    As $T$ is balanced, $Z(T)$ is a square zero-one matrix. Each non-zero summand of $\per{Z(T)}$ corresponds to a
    perfect matching, as it corresponds to a bijection between the two colour classes $B$ and $W$ of $G(T)$ (determined
    by the downward- and upward-pointing triangles of $T$).  Hence, $\per{Z(T)}$ enumerates the perfect matchings of
    $G(T)$, and thus the tilings of $T$.
\end{proof}

Recall that the \emph{determinant} of an $n \times n$ matrix $M$ is given by
\[
    \det{M} := \sum_{\sigma \in \PS_n} \sgn{\sigma} \prod_{i=1}^{n}  M_{(i, \sigma(i))},
\]
where $\sgn{\sigma}$ is the signature (or sign) of the permutation $\sigma$.
We take the
convention that the permanent and determinant of a $0 \times 0$ matrix its one.

By the proof of Proposition~\ref{pro:per-enum}, each lozenge tiling $\tau$ corresponds to a perfect matching $\pi$ of $G(T)$, that is, a bijection
$\pi: B \to W$. Considering $\pi$ as a permutation on $\#\uptri(T) = \#\dntri (T)$ letters, it is natural to assign a sign to each lozenge tiling
using the signature of the permutation $\pi$.

\begin{definition} \label{def:pm-sign}
    Let $T \subset \mathcal{T}_d$ be a non-empty balanced subregion. Then we define the \emph{perfect matching sign}
    of a lozenge tiling $\tau$ of $T$ as $\msgn{\tau} := \sgn{\pi}$, where $\pi \in \PS_{\#\uptri(T)}$ is the perfect
    matching determined by $\tau$.
\end{definition}

It follows that the determinant of $Z(T)$ gives an enumeration of the \emph{perfect matching signed lozenge tilings} of $T$.

\begin{theorem} \label{thm:pm-matrix}
     Let $T \subset \mathcal{T}_d$ be a non-empty balanced subregion. Then the perfect matching signed lozenge tilings of $T$
    are enumerated by $\det{Z(T)}$, that is,
    \[
        \sum_{\tau \text{tiling of } T}  \msgn{\tau} = \det Z(T).
    \]
\end{theorem}~

\begin{example} \label{exa:Z-matrix}
    Consider the triangular region $T = T_6(x^3, y^4, z^5)$, as seen in the first picture of Figure~\ref{fig:three-rotations} below.
    Then $Z(T)$ is the $11 \times 11$ matrix
    \[
        Z(T) =
        \left[
            \begin{array}{ccccccccccc}
                1&1&0&0&0&0&0&0&0&0&0\\
                0&1&1&0&0&0&0&0&0&0&0\\
                0&0&1&1&0&0&0&0&0&0&0\\
                1&0&0&0&1&0&0&0&0&0&0\\
                0&1&0&0&1&1&0&0&0&0&0\\
                0&0&1&0&0&1&1&0&0&0&0\\
                0&0&0&1&0&0&1&1&0&0&0\\
                0&0&0&0&1&0&0&0&1&0&0\\
                0&0&0&0&0&1&0&0&1&1&0\\
                0&0&0&0&0&0&1&0&0&1&1\\
                0&0&0&0&0&0&0&1&0&0&1
            \end{array}
        \right].
    \]
    We note that $\per Z(T) = \det{Z(T)} = 10$. Thus, $T$ has exactly $10$ lozenge tilings, all of which have the same sign.
    We derive a theoretical explanation for this fact in the following section.
\end{example}~

\subsection{Families of non-intersecting lattice paths}\label{sub:nilp}~\par

We follow~\cite[Section~5]{CEKZ} (similarly,~\cite[Section~2]{Fi}) in order to associate to a subregion $T \subset \mathcal{T}_d$ a
finite set  $L(T)$ that can be identified with a subset of  the lattice $\ZZ^2$.  Abusing notation, we refer to $L(T)$ as a
sub-lattice of $\ZZ^2$. We then translate lozenge tilings of $T$ into families of non-intersecting lattice paths on $L(T)$.

We first construct $L(T)$ from $T$.  Place a vertex at the midpoint of the edge of each triangle of $T$  that is parallel to the upper-left
boundary of the triangle $\mathcal{T}_d$.  These vertices form $L(T)$. We will consider paths in $L(T)$. There we think of rightward motion parallel to the
bottom edge of $\mathcal{T}_d$ as ``horizontal'' and downward motion parallel to the upper-right edge of $\mathcal{T}_d$ as ``vertical'' motion.  If one simply orthogonalises
$L(T)$ with respect to the described ``horizontal'' and ``vertical'' motions, then we can consider $L(T)$ as a finite sub-lattice of $\ZZ^2$.
As we can translate $L(T)$ in $\ZZ^2$ and not change its properties, we may assume that the vertex associated to the
lower-left triangle of $\mathcal{T}_d$ is the origin.  Notice that each vertex of $L(T)$  is on the upper-left edge of an upward-pointing triangle of $\mathcal{T}_d$ (even if
this triangle is not present in $T$). We use the monomial label of this upward-pointing triangle to specify a vertex of $L(T)$. Under this identification the mentioned
orthogonalization of $L(T)$ moves the vertex associated to the monomial $x^a y^b z^{d-1-(a+b)}$ in $L(T)$  to the point $(d-1-b, a)$ in $\ZZ^2$.

We next single out special vertices of $L(T)$. We label the vertices of $L(T)$ that are only on upward-pointing triangles in $T$, from smallest to largest
in the reverse-lexicographic order, as $A_1, \ldots, A_m$.  Similarly, we label the vertices of $L(T)$ that are only on downward-pointing triangles in $T$,
again from smallest to largest in the reverse-lexicographic order, as $E_1, \ldots, E_n$.  See Figure~\ref{fig:build-nilp}(i).  We note that there are an equal
number of vertices  $A_1, \ldots, A_m$ and $E_1, \ldots, E_n$ if and only if the region $T$ is balanced.  This follows from the fact  these vertices are
precisely the vertices of $L(T)$ that are in exactly one unit triangle of $T$.

A \emph{lattice path} in a lattice $L \subset \ZZ^2$ is a finite sequence of vertices of $L$  so that all single steps move either to the right or down.
Given any vertices $A, E \in \ZZ^2$, the number of lattice paths in $\ZZ^2$ from $A$ to $E$ is a binomial coefficient.  In fact, if $A$ and $E$ have
coordinates $(u,v), (x,y) \in \ZZ^2$ as above, there are $\binom{x-u+v-y}{x-u}$ lattice paths from $A$ to $E$ as each path has $x-u + v-y$ steps
and $x-u \geq 0$ of these must be horizontal steps.

Using the above identification of $L(T)$ as a sub-lattice of $\ZZ^2$, a \emph{lattice path} in $L(T)$ is a finite sequence of vertices of $L(T)$
so that all single steps move either to the East or to the Southeast. The \emph{lattice path matrix}
of $T$ is the $m \times n$ matrix  $N(T)$ with entries $N(T)_{(i,j)}$ defined by
\[
    N(T)_{(i,j)} = \# \text{lattice paths in $\ZZ^2$ from $A_i$ to $E_j$}.
\]
Thus,  the entries of $N(T)$ are binomial coefficients.

Next we consider several lattice paths simultaneously. A \emph{family of non-intersecting lattice paths}
is a finite collection of lattice paths such that no two lattice paths have any points in common.  We call a family of non-intersecting
lattice paths \emph{minimal}
if every path takes vertical steps before it takes horizontal steps, whenever possible.  That is, every time a horizontal step is followed by a vertical step, then replacing
these with a vertical step followed by a horizontal step would cause paths in the family to intersect.

Assume now that the subregion $T$ is balanced, so $m = n$. Let $\Lambda$ be a family of $m$ non-intersecting lattice paths in $L(T)$ from
$A_1, \ldots, A_m$ to $E_1, \ldots, E_m$. Then $\Lambda$ determines a permutation $\lambda \in \PS_m$ such that the path in $\Lambda$ that
begins at $A_i$ ends at $E_{\lambda(i)}$.

Now we are ready to apply a beautiful theorem relating enumerations of signed families of non-intersecting lattice paths and
determinants.  In particular, we use a theorem first given by Lindstr\"om in~\cite[Lemma~1]{Li} and stated independently
in~\cite[Theorem~1]{GV} by Gessel and Viennot.  Stanley gives a very nice exposition of the topic in~\cite[Section~2.7]{Stanley-2011}.

\begin{theorem}{\cite[Lemma~1]{Li} \& \cite[Theorem~1]{GV}} \label{thm:lgv}
    Assume $T \subset \mathcal{T}_d$ is a non-empty balanced subregion with identified lattice points
    $A_1, \ldots, A_m, E_1, \ldots, E_m \in L(T)$ as above. Then
    \[
        \det{N(T)} = \sum_{\lambda \in \PS_m} \sgn(\lambda) \cdot P^+_\lambda(A\rightarrow E),
    \]
    where, for each permutation $\lambda \in \PS_m$, $P^+_\lambda(A \rightarrow E)$ is the number of families of non-intersecting
    lattice paths with paths in $L(T)$ going from $A_i$ to $E_{\lambda(i)}$.
\end{theorem}

We now use a well-know bijection between lozenge tilings of $T$ and families of non-intersecting lattice paths from $A_1, \ldots, A_m$ to $E_1, \ldots, E_m$;
see, e.g., the survey~\cite{Pr}.  Let $\tau$ be a lozenge tiling of $T$.  Using the lozenges of $\tau$ as a guide,
we connect each pair of vertices of $L(T)$ that occur on a single lozenge.  This generates a family of non-intersecting lattice
paths $\Lambda$ of $L(T)$ corresponding to $\tau$.  See Figures~\ref{fig:build-nilp}(ii) and~(iii) for the overlayed image and the family of non-intersecting
lattice paths by itself, respectively.

\begin{figure}[!ht]
    \begin{minipage}[b]{0.32\linewidth}
        \centering
        \includegraphics[scale=1]{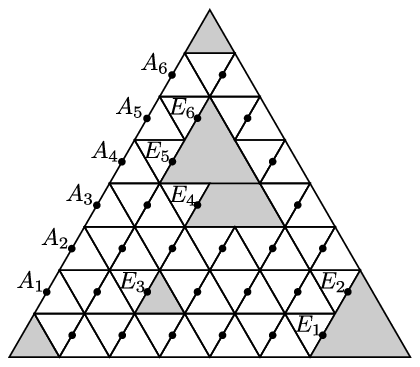}\\
        \emph{(i) The sub-lattice $L(T)$.}
    \end{minipage}
    \begin{minipage}[b]{0.32\linewidth}
        \centering
        \includegraphics[scale=1]{figs/build-nilp-2}\\
        \emph{(ii) The overlayed image.}
    \end{minipage}
    \begin{minipage}[b]{0.32\linewidth}
        \centering
        \includegraphics[scale=1]{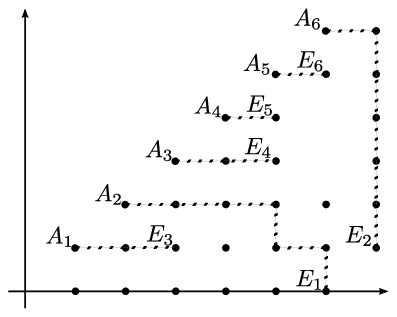}\\
        \emph{(iii) The family $\Lambda$.}
    \end{minipage}
    \caption{The family of non-intersecting lattice paths $\Lambda$ associated to the tiling $\tau$ in Figure~\ref{fig:triregion-intro}.}
    \label{fig:build-nilp}
\end{figure}

This bijection provides another way for assigning a sign to a lozenge tiling, this time using the signature of the permutation $\lambda$.

\begin{definition} \label{def:nilp-sign}
    Let $T \subset \mathcal{T}_d$ be a non-empty balanced subregion as above, and let $\tau$ be a lozenge tiling of $T$.
    Then we define the \emph{lattice path sign} of $\tau$ as $\lpsgn{\tau} := \sgn{\lambda}$, where $\lambda \in \PS_m$ is
    the permutation such that, for each $i$, the lattice path determined by $\tau$ that starts at $A_i$ ends at $E_{\lambda (i)}$.
\end{definition}

It follows that the determinant of $N(T)$ gives an enumeration of the \emph{lattice path signed lozenge tilings of $T$}.

\begin{theorem} \label{thm:nilp-matrix}
    Let $T \subset \mathcal{T}_d$ be a non-empty balanced subregion.  Then the lattice path signed lozenge tilings of $T$
    are enumerated by $\det{N(T)}$, that is,
    \[
        \sum_{\tau \text{tiling of } T} \lpsgn{\tau} = \det{N(T)}.
    \]
\end{theorem}

\begin{remark} \label{rem:rotations}
    Notice that we can use the above construction to assign, for each subregion $T$,  three (non-trivially) different lattice path matrices.
    The matrix $N(T)$ from Theorem~\ref{thm:nilp-matrix} is one of these matrices, and the other two are the $N(\cdot)$ matrices of the
    $120^{\circ}$ and $240^{\circ}$ rotations of $T$.  See Figure~\ref{fig:three-rotations} for an example.

    \begin{figure}[!ht]
        \includegraphics[scale=1]{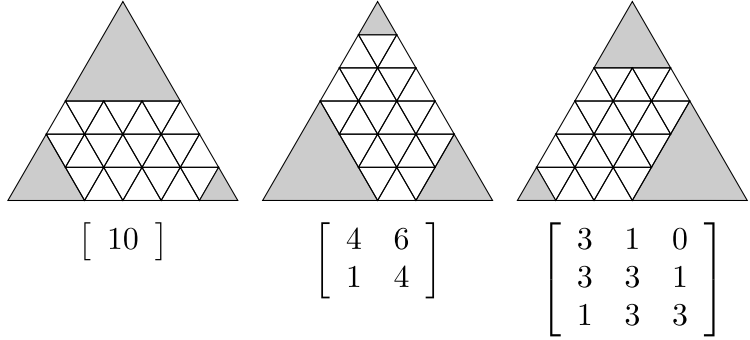}
        \caption{The triangular region $T_6(x^3, y^4, z^5)$ and its rotations, along with their lattice path matrices. }
        \label{fig:three-rotations}
    \end{figure}
\end{remark}

\section{Resolution of punctures} \label{sec:resolution}

In the previous section we associated  two different signs, the perfect matching sign and the lattice path sign,  to each lozenge tiling of a
balanced region $T$.  In the case where $T$ is a triangular region, we demonstrate in this section that the
signs are equivalent, up to a scaling factor dependent only on $T$. In particular, Theorem~\ref{thm:detZN} states that
$|\det{Z(T)}| = |\det{N(T)}|$. In order to prove this result, we introduce a new method that we call resolution of a puncture. Throughout this section $T$
is a tileable triangular region. In particular,  $T$ is balanced.

\subsection{The construction}\label{sub:rez}~\par

Our first objective is to describe a construction that  removes a puncture from a triangular region, relative to some tiling, in a controlled fashion. More precisely, starting from a given region with a puncture, we produce a larger triangular region without this puncture.

We begin by considering the special case, in which we assume that $T \subset \mathcal{T}_d$ has at least one puncture, call it $\mathcal{P}$, that is not overlapped
by any other puncture of $T$. Let $\tau$ be some lozenge tiling of $T$, and denote by $k$ the side length of
$\mathcal{P}$. Informally, we will replace $T$ by a triangular region in $\mathcal{T}_{d + 2k}$, where the place of
the puncture $\mathcal{P}$ of $T$ is taken by a tiled regular hexagon of side length $k$ and three corridors to the outer
vertices of $\mathcal{T}_{d + 2k}$ that are all part of the new region.

\begin{figure}[!ht]
    \begin{minipage}[b]{0.48\linewidth}
        \centering
        \includegraphics[scale=1]{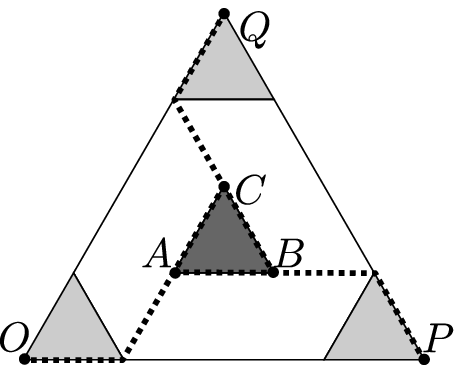}\\
        \emph{(i)  The splitting chains.}
    \end{minipage}
    \begin{minipage}[b]{0.48\linewidth}
        \centering
        \includegraphics[scale=1]{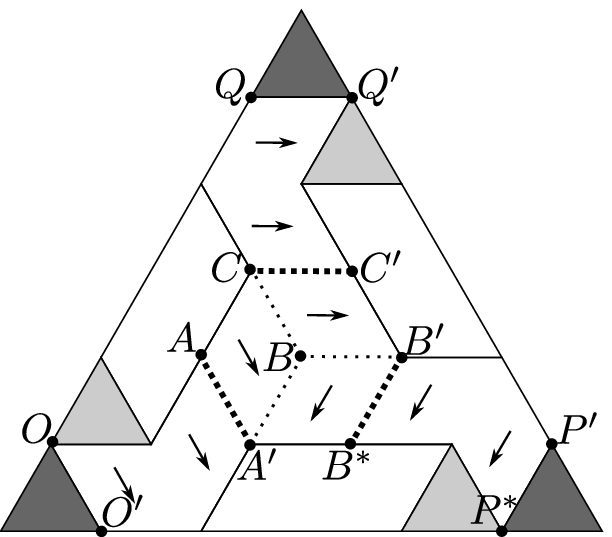}\\
        \emph{(ii) The resolution $T'$.}
    \end{minipage}
    \caption{The abstract resolution of a puncture.}
    \label{fig:resolve-abstract}
\end{figure}

As above, we label the vertices of $\mathcal{T}_d$ such that the label of each unit triangle is the greatest common
divisor of its vertex labels. For ease of reference, we denote the lower-left, lower-right, and top vertex of the
puncture $\mathcal{P}$ by $A, B$, and $C$, respectively. Similarly, we denote the lower-left, lower-right, and top vertex
of $\mathcal{T}_d$ by $O, P$, and $Q$, respectively. Now we select three chains of unit edges such that each edge
is either in $T$ or on the boundary of a puncture of $T$. We start by choosing chains connecting $A$ to $O$, $B$
to $P$, and $C$ to $Q$, respectively, subject to the following conditions:
\begin{itemize}
    \item The chains do not cross, that is, do not share any vertices.
    \item There are no redundant edges, that is, omitting any unit edge destroys the connection between the desired end points of a chain.
    \item There are no moves to the East or Northeast on the lower-left chain $OA$.
    \item There are no moves to the West or Northwest on the lower-right chain $PB$.
    \item There are no moves to the Southeast or Southwest on the top chain $CQ$.
\end{itemize}
For these directions we envision a particle that starts at a vertex of the puncture and moves on a chain to the corresponding corner vertex of $\mathcal{T}_d$.

Now we connect the chains $OA$ and $CQ$ to a chain of unit edges $OACQ$ by using the Northeast edge of $\mathcal{P}$. Similarly we
connect the chains $OA$ and $BP$ to a chain $OABP$ by using the horizontal edge of $\mathcal{P}$, and we connect
$PB$ and $CQ$ to the chain $PBCQ$ by using the Northwest side of $\mathcal{P}$. These three chains subdivide
$\mathcal{T}_d$ into four regions. Part of the boundary of three of these regions is an edge of $\mathcal{T}_d$.
The fourth region, the central one, is the area of the puncture $\mathcal{P}$. See Figure~\ref{fig:resolve-abstract}(i)
for an illustration.

Now consider $T \subset \mathcal{T}_d$ as embedded into $\mathcal{T}_{d+ 2k}$ such that the original region
$\mathcal{T}_d$ is identified with the triangular region $T_{d+2k} (x^k y^k)$. Retain the names $A, B, C, O, P$, and $Q$
for the specified vertices of $T$ as above. We create new chains of unit edges in $\mathcal{T}_{d+ 2k}$.

First, multiply each vertex in the chain $PBCQ$ by $\frac{z^k}{y^k}$ and connect the resulting vertices to a chain
$P'B'C'Q'$ that is parallel to the chain $PBCQ$. Here $P', B', C'$, and $Q'$ are the images of $P, B, C$, and $Q$ under
the multiplication by $\frac{z^k}{y^k}$. Informally, the chain $P'B'C'Q'$ is obtained by moving the chain $PBCQ$ just
$k$ units to the East.

Second, multiply each vertex in the chain $OA$ by $\frac{z^k}{x^k}$ and connect the resulting vertices to a chain $O'A'$
that is parallel to the chain $OA$. Here $A'$ and $O'$ are the points corresponding to $A$ and $O$. Informally the chain
$O'A'$ is obtained by moving the chain $OA$ just $k$ units to the Southeast.

Third, multiply each vertex in the chain $P'B'$ by $\frac{y^k}{x^k}$ and connect the resulting vertices to a chain
$P^*B^*$ that is parallel to the chain $P'B'$, where $P^*$ and $B^*$ are the images of $P'$ and $B'$, respectively.
Thus, $P^*B^*$ is $k$ units to the Southwest of the chain $P'B'$. Connecting $A'$ and $B^*$ by horizontal edges, we
obtain a chain $O'A'B^*P^*$ that has the same shape as the chain $OABP$. See Figure~\ref{fig:resolve-abstract}(ii) for
an illustration.

We are ready to describe the desired triangular region $T' \subset \mathcal{T}_{d+2k}$ along with a tiling. Place
lozenges and punctures in the region bounded by the chain $OACQ$ and the Northeast boundary of $\mathcal{T}_{d+2k}$ as
in the corresponding region of $T$. Similarly place lozenges and punctures in the region bounded by the chain $P'B'C'Q'$
and the Northwest boundary of $\mathcal{T}_{d+2k}$ as in the corresponding region of $T$ that is bounded by $PBCQ$.
Next, place lozenges and punctures in the region bounded by the chain $O'A'B^*P^*$ and the horizontal boundary of
$\mathcal{T}_{d+2k}$ as in the exterior region of $T$ that is bounded by $OABP$. Observe that corresponding vertices of
the parallel chains $BCQ$ and $B'C'Q'$ can be connected by horizontal edges. The region between two such edges that are
one unit apart is uniquely tileable. This gives a lozenge tiling for the region between the two chains. Similarly, the
corresponding vertices of the parallel chains $OAC$ and $O'A'C'$ can be connected by Southeast edges. Respecting these
edges gives a unique lozenge tiling for the region between the chains $OAC$ and $O'A'C'$. In a similar fashion, the
corresponding vertices of the parallel chains $P'B'$ and $P^*B^*$ can be connected by Southwest edges, which we use as a
guide for a lozenge tiling of the region between the two chains. Finally, the rhombus with vertices $A', B^*, B'$, and
$B$ admits a unique lozenge tiling. Let $\tau'$ the union of all the lozenges we placed in $\mathcal{T}_{d+2k}$, and
denote by $T'$ the triangular region that is tiled by $\tau'$. Thus, $T' \subset \mathcal{T}_{d+2k}$ has a puncture of
side length $k$ at each corner of $\mathcal{T}_{d+2k}$. See Figure~\ref{fig:resolve-simple} for an illustration of this.
We call the region $T'$ with its tiling $\tau'$ a \emph{resolution of the puncture $\mathcal{P}$ in $T$ relative to $\tau$}
or, simply, a \emph{resolution of $\mathcal{P}$}.

Observe that the tiles in $\tau'$ that were not carried over from the tiling $\tau$ are in the region that is the union
of the regular hexagon with vertices $A, A', B^*, B', C'$ and $C$ and the regions between the parallel chains $OA$ and
$O'A'$, $CQ$ and $C'Q'$ as well as $P'B'$ and $P^*B^*$. We refer to the latter three regions as the \emph{corridors} of
the resolution. Furthermore, we call the chosen chains $OA$, $PB$, and $CQ$ the \emph{splitting chains}
of the resolution. The resolution blows up each splitting chain to a corridor of width $k$.

\begin{figure}[!ht]
    \begin{minipage}[b]{0.48\linewidth}
        \centering
        \includegraphics[scale=1]{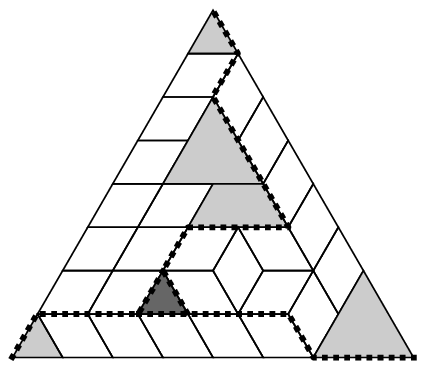}\\
        \emph{(i)  The selected lozenge and puncture edges.}
    \end{minipage}
    \begin{minipage}[b]{0.48\linewidth}
        \centering
        \includegraphics[scale=1]{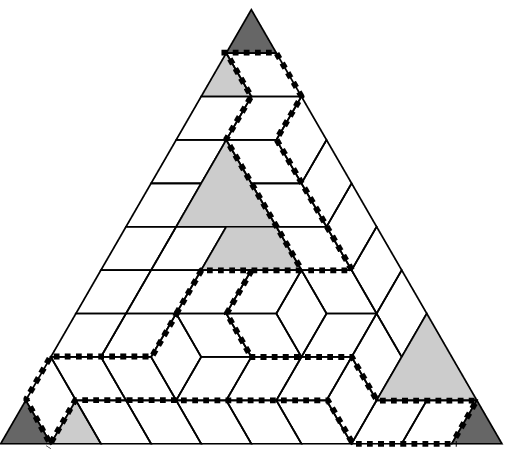}\\
        \emph{(ii) The resolution $T'$ with tiling $\tau'$.}
    \end{minipage}
    \caption{A resolution of the puncture associated to $x y^4 z^2$, given the tiling $\tau$ in Figure~\ref{fig:triregion-intro} of $T$.}
    \label{fig:resolve-simple}
\end{figure}

Finally, in order to deal with an arbitrary puncture  suppose a puncture $\mathcal{P}$ in $T$ is overlapped by another puncture of $T$. Then we cannot resolve
$\mathcal{P}$ using the above technique directly as it would result in a non-triangular region. Thus, we adapt the construction. Since
$T$ is balanced, $\mathcal{P}$ is overlapped by exactly one puncture of $T$ (see Theorem~\ref{thm:tileable}). Let $U$ be
the smallest monomial subregion of $T$ that contains both punctures. We call $U$ the \emph{minimal covering region}
of the two punctures. It is is uniquely tileable, and we resolve the puncture $U$ of $T \setminus U$. Notice that the
lozenges inside $U$ are lost during resolution. However, since $U$ is uniquely tileable, they are recoverable from the
two punctures of $T$ in $U$. See Figure~\ref{fig:resolve-family} for an illustration.

\begin{figure}[!ht]
    \begin{minipage}[b]{0.48\linewidth}
        \centering
        \includegraphics[scale=1]{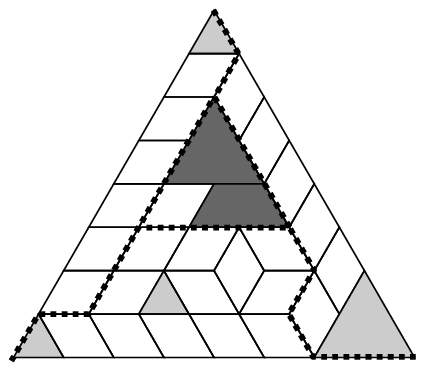}\\
        \emph{(i)  The selected lozenge and puncture edges.}
    \end{minipage}
    \begin{minipage}[b]{0.48\linewidth}
        \centering
        \includegraphics[scale=0.75]{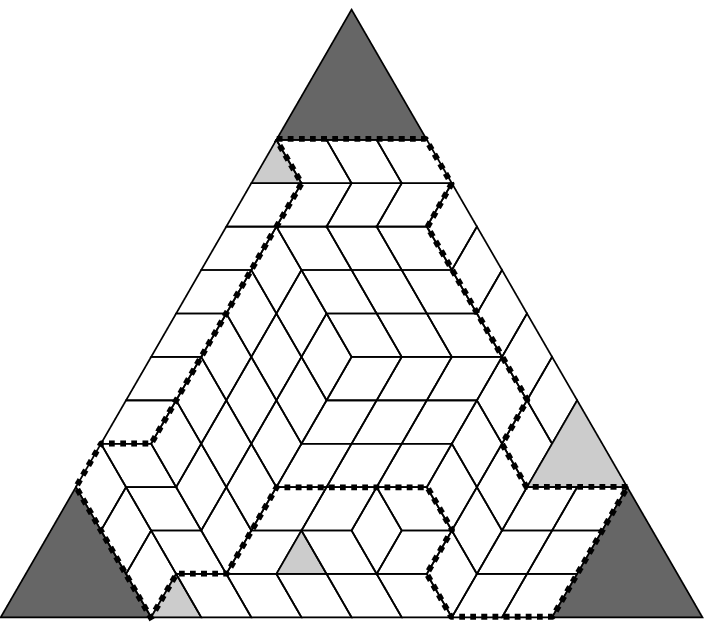}\\
        \emph{(ii) The resolution $T'$ with tiling $\tau'$.}
    \end{minipage}
    \caption{Resolving overlapping punctures, given the tiling in Figure~\ref{fig:triregion-intro}.}
    \label{fig:resolve-family}
\end{figure}

\subsection{Cycles of lozenges}\label{sub:cyc}~\par

We now introduce another concept. It will help us to  analyze the changes when resolving a puncture.

Let $\tau$ be some tiling of  a triangular region $T$. An \emph{$n$-cycle (of lozenges)}
$\sigma$ in $\tau$ is an ordered collection of distinct lozenges $\ell_1, \ldots, \ell_n$ of $\tau$ such that the
downward-pointing triangle of $\ell_i$ is adjacent to the upward-pointing triangle of $\ell_{i+1}$ for $1 \leq i < n$
and the downward-pointing triangle of $\ell_n$ is adjacent to the upward-pointing triangle of $\ell_1$. The smallest
cycle of lozenges is a three-cycle; see Figure~\ref{fig:three-cycle}.

\begin{figure}[!ht]
    \includegraphics[scale=1]{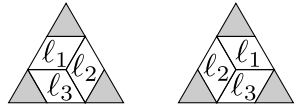}
    \caption{$T_3(x^2, y^2, z^2)$ has two tilings, both are three-cycles of lozenges.}
    \label{fig:three-cycle}
\end{figure}

Let $\sigma = \{\ell_1, \ldots, \ell_n\}$ be an $n$-cycle of lozenges in the tiling $\tau$ of $T$. If we replace the
lozenges in $\sigma$ be the $n$ lozenges created by adjoining the downward-pointing triangle of $\ell_i$ with the
upward-pointing triangle of $\ell_{i+1}$ for $1 \leq i < n$ and the downward-pointing triangle of $\ell_n$ with the
upward-pointing triangle of $\ell_1$, then we get a new tiling $\tau'$ of $T$. We call this new tiling the \emph{twist of $\sigma$}
in $\tau$. The two three-cycles in Figure~\ref{fig:three-cycle} are twists of each other. See
Figure~\ref{fig:cycle-twist} for another example of twisting a cycle. A puncture is \emph{inside}
the cycle $\sigma$ if the lozenges of the cycle fully surround the puncture. In Figure~\ref{fig:cycle-twist}(i), the
puncture associated to $x y^4 z^2$ is inside the cycle $\sigma$ and all other punctures of $T$ are not inside the cycle $\sigma$.

\begin{figure}[!ht]
    \begin{minipage}[b]{0.48\linewidth}
        \centering
        \includegraphics[scale=1]{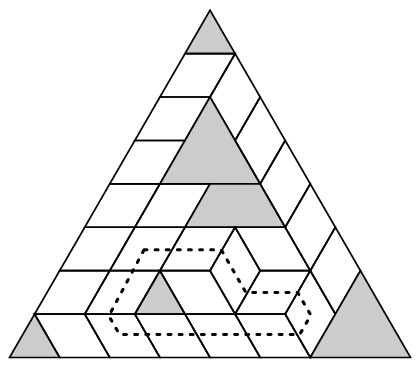}\\
        \emph{(i) A $10$-cycle $\sigma$.}
    \end{minipage}
    \begin{minipage}[b]{0.48\linewidth}
        \centering
        \includegraphics[scale=1]{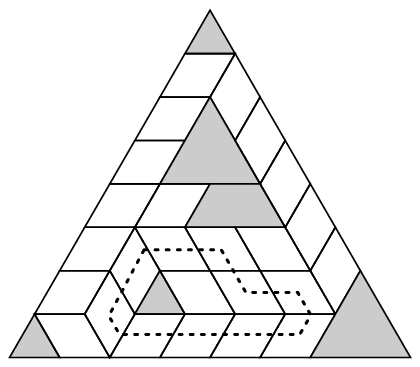}\\
        \emph{(ii) The twist of $\sigma$ in $\tau$.}
    \end{minipage}
    \caption{A $10$-cycle $\sigma$ in the tiling $\tau$ (see Figure~\ref{fig:triregion-intro}(ii)) and its twist.}
    \label{fig:cycle-twist}
\end{figure}

Recall that the perfect matching sign of a tiling $\tau$ is denoted by $\msgn{\tau}$ (see Definition~\ref{def:pm-sign}).

\begin{lemma} \label{lem:twist-sign}
    Let $\tau$ be a lozenge tiling of a triangular region $T = T_d(I)$, and let $\sigma$ be an $n$-cycle of lozenges in $\tau$.
    Then the twist $\tau'$ of $\sigma$ in $\tau$ satisfies $\msgn{\tau'} = (-1)^{n-1}\msgn{\tau}$.
\end{lemma}
\begin{proof}
    Let $\pi$ and $\pi'$ be the perfect matching permutations associated to $\tau$ and $\tau'$, respectively (see Definition \ref{def:pm-sign}).  Without loss of generality,
    assume each lozenge $\ell_i$ in $\sigma$ corresponds to the upward- and downward-pointing triangles labeled $i$.  As $\tau'$ is a twist of $\tau$ by $\sigma$,
    then $\pi'(i) = i+1$ for $1 \leq i < n$ and $\pi'(n) = 1$.  That is, $\pi' = (1, 2, \ldots, n) \cdot \pi$, as permutations.  Hence,
    $\msgn{\tau'} = (-1)^{n-1}\msgn{\tau}$.
\end{proof}

\subsection{Resolutions, cycles of lozenges, and signs}\label{sub:rez-cyc}~\par

Now we are going to establish the equivalence of the perfect matching and the lattice path sign of a lozenge tiling. We begin by describing the modification of a cycle of lozenges when a puncture is resolved.

We first need a definition.  It
uses the starting and end points of lattice paths $A_1,\ldots,A_m$ and $E_1,\ldots,E_m$, as introduced at the beginning
of Subsection~\ref{sub:nilp}.

The \emph{$E$-count}
of a cycle is the number of lattice path end points $E_j$ ``inside'' the cycle. Alternatively, this
can be seen as the sum of the side lengths of the non-overlapping punctures plus the sum of the side lengths of the
minimal covering regions of pairs of overlapping punctures. For example, the cycles shown in
Figure~\ref{fig:three-cycle} have $E$-counts of zero, the cycles shown in Figure~\ref{fig:cycle-twist} have $E$-counts
of $1$, and the (unmarked) cycle going around the outer edge of the tiling shown in Figure~\ref{fig:cycle-twist}(i) has an
$E$-count of $1 + 3 = 4$.

Now we describe the change of a cycle surrounding a puncture when this puncture is resolved.

\begin{lemma} \label{lem:cycle-res}
    Let $\tau$ be a lozenge tiling of $T = T_d(I)$, and let $\sigma$ be an $n$-cycle of lozenges in $\tau$.
    Suppose $T$ has a puncture $P$ (or a minimal covering region of a pair of overlapping punctures) with $E$-count $k$.
    Let $T'$ be a resolution of $P$ relative to $\tau$.  Then the resolution takes $\sigma$ to an $(n+kl)$-cycle of
    lozenges $\sigma'$ in the resolution, where $l$ is the number of times the splitting chains of the resolution
    cross the cycle $\sigma$ in $\tau$.  Moreover, $l$ is odd if and only if $P$ is inside $\sigma.$
\end{lemma}
\begin{proof}
    Fix a resolution $T' \subset \mathcal{T}_{d+2k}$ of $P$ with tiling $\tau'$ as induced by $\tau$.

    First, note that if $P$ is a minimal covering region of a pair of overlapping punctures, then any cycle of lozenges must
    avoid the lozenges present in $P$ as all such lozenges are forcibly chosen, i.e., immutable.  Thus, all lozenges of
    $\sigma$ are present in $\tau'$.

    The resolution takes the cycle $\sigma$ to a cycle $\sigma'$ by adding $k$ new lozenges for each unit edge of a lozenge in $\sigma$
    that belongs to a splitting chain.  More precisely, such an edge is expanded to $k+1$ parallel edges.  Any two consecutive
    edges form the opposite sides of a lozenge (see Figure~\ref{fig:cycle-insert}).  Thus, each time a splitting chain of the resolution
    crosses the cycle $\sigma$ we insert $k$ new lozenges.  As $l$ is the number of times the splitting chains of the resolution
    cross the cycle $\sigma$ in $\tau$, the resolution adds exactly $k l$ new lozenges to the extant lozenges of $\sigma$.
    Thus, $\sigma'$ is an $(n + k l)$-cycle of lozenges in $\tau'$.

    \begin{figure}[!ht]
        \includegraphics[scale=1]{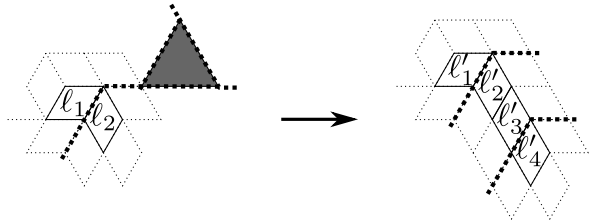}
        \caption{Expansion of a lozenge cycle at a crossing of a splitting chain.}
        \label{fig:cycle-insert}
    \end{figure}

    Since the splitting chains are going from $P$ to the boundary of the triangle $\mathcal{T}_d$, the splitting chains terminate
    outside the cycle.  Hence if the splitting chain crosses into the cycle, it must cross back out.  If $P$ is outside
    $\sigma$, then the splitting chains start outside $\sigma$, and so $l$ must be even.  On the other hand,
    if $P$ is inside $\sigma$, then the splitting chains start inside of $\sigma$, and so $l = 3 + 2j$, where $j$
    is the number of times the splitting chains cross \emph{into} the cycle.
\end{proof}

Let $\tau_1$ and $\tau_2$ be tilings of $T$, and let $\pi_1$ and $\pi_2$ be their respective perfect matching permutations.  Suppose
$\pi_2 = \rho \pi_1$, for some permutation $\rho$.  Write $\rho$ as a product of disjoint cycles whose length is at least two.
(Note that these cycles will be of length at least three.) Each factor corresponds to a  cycle of lozenges of $\tau_1$. If all these cycles are twisted we get $\tau_2$. We call these
lozenge cycles the \emph{difference cycles} of $\tau_1$ and $\tau_2$.

Using the idea of difference cycles, we characterise when two tilings have the same perfect matching sign.

\begin{corollary} \label{cor:msgn-cycle}
    Let $\tau$ be a lozenge tiling of $T = T_d(I)$, and let $\sigma$ be an $n$-cycle of lozenges in $\tau$.  Then
    the following statements hold.
    \begin{itemize}
        \item The $E$-count of $\sigma$ is even if and only if $n$ is odd.
        \item Two lozenge tilings of $T$ have the same perfect matching sign if and only if the sum of the
            $E$-counts of the difference cycles is even.
    \end{itemize}
\end{corollary}
\begin{proof}
    Suppose $T$ has $a$ punctures and pairs of overlapping punctures, $P_1, \ldots, P_a$, inside $\sigma$ that are
    \emph{not} in a corner, i.e., not associated to $x^k$, $y^k$, or $z^k$, for some $k$.  Let $j_i$ be the
    $E$-count of $P_i$.  Similarly, suppose $T$ has $b$ punctures and pairs of overlapping punctures, $Q_1, \ldots, Q_b$,
    outside $\sigma$ that are \emph{not} in a corner, i.e., not associated to $x^k$, $y^k$, or $z^k$, for some $k$.
    Let $k_i$ be the $E$-count of $Q_i$.

    If we resolve all of the punctures $P_1, \ldots, P_a, Q_1, \ldots, Q_b$, then $\sigma$ is taken to a  cycle
    $\sigma'$.  By Lemma~\ref{lem:cycle-res}, $\sigma'$ has length
    \[
        n' := n + (j_1 l_1 + \cdots + j_a l_a) + (k_1 m_1 + \cdots + k_b m_b),
    \]
    where the integers $l_1, \ldots, l_a$ are odd and the integers $m_1, \ldots, m_b$ are even.

    Denote the region obtained from $T$ by resolving its $a+b$ punctures by $T'$. After merging touching punctures, it becomes a hexagon.  By \cite[Theorem~1.2]{CGJL}, every tiling of $T'$ is thus obtained from any
    other tiling of $T'$ through a sequence of three-cycle twists, as in Figure~\ref{fig:three-cycle}.
    By Lemma~\ref{lem:twist-sign}, such twists do not change the perfect matching sign of the tiling,
    hence $n'$ is an odd integer.

    Since $n'$ is odd, $n' - (k_1 m_1 + \cdots + k_b m_b) = n + (j_1 l_1 + \cdots + j_a l_a)$ is also odd.
    Thus, $n$ is odd if and only if $j_1 l_1 + \cdots + j_a l_a$ is even.  Since the integers $l_1, \ldots, l_a$ are odd,
    we see that $j_1 l_1 + \cdots + j_a l_a$ is even if and only if an even number of the $l_i$ are odd, i.e., the sum
    $l_1 + \cdots + l_a$ is even.  Notice that this sum is the $E$-count of $\sigma$. Thus, claim (i) follows.

    Suppose two tilings $\tau_1$ and $\tau_2$ of $T$ have difference cycles $\sigma_1, \ldots, \sigma_p$.  Then
    by Lemma~\ref{lem:twist-sign}, $\msgn{\tau_2} = \sgn{\sigma_1} \cdots \sgn{\sigma_p} \msgn{\tau_1}$.  By claim (i), $\sigma_i$
    is a cycle of odd length if and only if the $E$-count of $\sigma_i$ is even.  Thus, $\sgn{\sigma_1} \cdots \sgn{\sigma_p} = 1$ if and only if
    an even number of the $\sigma_i$ have an odd $E$-count.  An even number of the $\sigma_i$ have an odd $E$-count if and only
    if the sum of the $E$-counts of $\sigma_1, \ldots, \sigma_p$ is even.  Hence, claim (ii) follows.
\end{proof}


Next, we describe the change of a lattice path permutation when  twisting a cycle of lozenges.  To this end  we single out certain punctures. We recursively define a puncture of $T \subset \mathcal{T}_d$ to be
a \emph{non-floating} puncture if it  touches the boundary of $ \mathcal{T}_d$ or if it overlaps or touches a non-floating puncture of $T$. Otherwise we call a puncture
a \emph{floating} puncture.

We also distinguish between \emph{preferred} and \emph{acceptable directions} on the splitting chains used for resolving
a puncture. Here we use again the perspective of a particle that starts at a vertex of the puncture and moves on a chain
to the corresponding corner vertex of $\mathcal{T}_d$. Our convention is:

\begin{itemize}
    \item On the lower-left chain the preferred direction are Southwest and West, the acceptable directions are Northwest and Southeast.
    \item On the lower-right chain the preferred directions are Southeast and East, the acceptable directions are Northeast and Southwest.
    \item On the top chain the preferred directions are Northeast and Northwest, the acceptable directions are East and West.
\end{itemize}

\begin{lemma} \label{lem:lpsgn-cycle}
    Let $\tau$ be a lozenge tiling of $T = T_d(I)$,  and let $\sigma$ be a cycle of lozenges in $\tau$.  Then the lattice
    path signs of $\tau$ and the twist of $\sigma$ in $\tau$ are the same if and only if the $E$-count of $\sigma$ is even.
\end{lemma}
\begin{proof}
    Suppose $T$ has $n$  floating punctures.  We proceed by induction on $n$ in five steps.

    \emph{Step $1$: The base case.}

    If $n = 0$, then every tiling of $T$ induces the same bijection $\{A_1,\ldots,A_m\} \to \{E_1,\ldots,E_m\}$. Thus, all tilings have the same lattice path sign.  Since $T$ has no
    floating punctures, $\sigma$  has an $E$-count of zero.  Hence, the claim is true if $n=0$.

    \emph{Step $2$: The set-up.}

    Suppose now that $n > 0$, and choose $P$ among the floating punctures and the minimal covering regions of two
    overlapping floating punctures of $T$ as the one that covers the upward-pointing unit triangle of $\mathcal{T}_d$
    with the smallest monomial label. Let $s > 0$ be the side length of $P$, and let $k$ be the $E$-count of $\sigma$.
    Furthermore, let $\upsilon$ be the lozenge tiling of $T$ obtained as twist of $\sigma$ in $\tau$. Both, $\tau$ and
    $\upsilon$, induce bijections $\{A_1,\ldots,A_m\} \to \{E_1,\ldots,E_m\}$, and we denote by $\lambda \in \PS_m$ and
    $\mu \in \PS_m$ the corresponding lattice path permutations, respectively. We have to show $\lpsgn \tau = (-1)^k \lpsgn \upsilon$, that is,
    \[
        \sgn \lambda = (-1)^k \sgn \mu.
    \]

    \emph{Step $3$: Resolutions.}

    We resolve $P$ relative to the tilings $\tau$ and $\upsilon$, respectively. For the resolution of $P$ relative to
    $\tau$, choose the splitting chains so that each unit edge has a preferred direction, except possibly the unit edges on
    the boundary of a puncture of $T$;  this is always possible. By our choice of $P$, no other floating punctures are to the
    lower-right of $P$. It follows that no edge on the lower-right chain crosses a lattice path, except possibly at the end
    of the lattice path.

    For the resolution of $P$ relative to $\upsilon$, use the splitting chains described in the previous paragraph, except
    for the edges that cross the lozenge cycle $\sigma$. They have to be adjusted since these unit edges disappear when
    twisting $\sigma$. We replace each such unit edge by a unit edge in an acceptable direction followed by a unit edge in a
    preferred direction so that the result has the same starting and end point as the unit edge they replace. Note that this
    is always possible and that this determines the replacement uniquely. The new chains meet the requirements on splitting
    chains.

    Using these splitting chains we resolve the puncture $P$ relative to $\tau$ and $\upsilon$, respectively. The result is
    a triangular region $T' \subset \mathcal{T}_{d+2s}$ with induced tilings $\tau'$ and $\upsilon'$, respectively. Denote
    by $\sigma'$ the extension of the cycle $\sigma$ in $T'$ (see Lemma~\ref{lem:cycle-res}). Since $\tau$ and $\upsilon$
    differ exactly on the cycle $\sigma$ and the splitting chains were chosen to be the same except on $\sigma$, it follows
    that twisting $\sigma'$ in $\tau'$ results in the tiling $\upsilon'$ of $T'$.

      \begin{figure}[!ht]
        \includegraphics[scale=1]{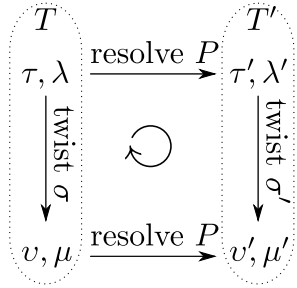}
        \caption{The commutative diagram used in the proof of Lemma~\ref{lem:lpsgn-cycle}.}
        \label{fig:res-comm-diag}
    \end{figure}

    \emph{Step $4$: Lattice path permutations.}

    Now we compare the signs of $\lambda, \mu \in \PS_m$ with the signs of  $\lambda'$ and $\mu'$, the lattice path permutations
    induced by the tilings $\tau'$ and $\upsilon'$ of $T'$, respectively.

    First, we compare the starting and end points of lattice paths in $T$ and $T'$. Resolution of the puncture identifies
    each starting and end point in $T$ with one such point in $T'$. We refer to these points as the \emph{old} starting and
    end points in $T'$. Note that the end points on the puncture $P$ correspond to the end points on the puncture in the
    Southeast corner of $T'$. The starting points in $T$ that are on one of the splitting chains used for resolving $P$
    relative to $\tau$ and $\upsilon$ are the same. Assume there are $t$ such points. After resolution, each point gives
    rise to a new starting and end point in $T'$. Both are connected by a lattice path that is the same in both resolutions
    of $P$. Hence, in order to compare the signs of the permutations $\lambda'$ and $\mu'$ on $m+t$ letters, it is enough to
    compare the lattice paths between the old starting and end points in both resolutions.

    Retain for these points the original labels used in $T$. Using this labeling, the lattice paths induce permutations
    $\tilde{\lambda}$ and $\tilde{\mu}$ on $m$ letters. Again, this is the same process in both resolutions. It follows that
    \begin{equation}\label{eq:compare-res-signs}
        \sgn (\tilde{\lambda}) \cdot \lpsgn (\tau') = \sgn (\tilde{\mu}) \cdot \lpsgn (\upsilon').
    \end{equation}

    Assume now that $P$ is a puncture. Then the end points on $P$ are indexed by $s$ consecutive integers. Since we retain
    the labels, the same indices label the end points on the puncture in the Southeast corner of $T'$. The end points on $P$
    correspond to the points in $T'$ whose labels are obtained by multiplying by $x^s y^s$. Consider now the case, where all
    edges in the lower-right splitting chain in $T$ are in preferred directions. Then the lattice paths induced by $\tau'$
    connect each point in $T'$ that corresponds to an end point on $P$ to the end point in the Southeast corner of $T'$ with
    the same index. Thus, $\sgn (\lambda) = \sgn (\tilde{\lambda})$. Next, assume that there is exactly one edge in
    acceptable direction on the lower-right splitting chain of $T$. If this direction is Northeast, then the $s$ lattice
    paths passing through the points in $T'$ corresponding to the end points on $P$ are moved one unit to the North. If the
    acceptable direction was Southwest, then the edge in this direction leads to a shift of these paths by one unit to the
    South. In either case, this shift means that the paths in $T$ and $T'$ connect to end points that differ by $s$
    transpositions, so $\sgn (\tilde{\lambda})= (-1)^{s} \sgn (\lambda)$. More generally, if $j$ is the number of unit edges
    on the lower-right splitting chain of $T$ that are in acceptable directions, then
    \[
        \sgn (\tilde{\lambda}) = (-1)^{js} \sgn (\lambda).
    \]

    Next, denote by $c$ the number of unit edges on the lower-right splitting chain that have to be adjusted when twisting
    $\sigma$. Since each of these edges is replaced by an edge in a preferred and one edge in an acceptable direction, after
    twisting the lower-right splitting chain in $T$ has exactly $j+c$ unit edges in acceptable directions. It follows as
    above that
    \[
        \sgn (\tilde{\mu}) = (-1)^{(j+c)s} \sgn (\mu).
    \]
    Since a unit edge on the splitting chain has to be adjusted when twisting if and only if it is shared by two
    consecutive lozenges in the cycle $\sigma$, the number $c$ is even if and only if the puncture $P$ is outside
    $\sigma$.

   Moreover, as the puncture $P$ has been resolved in $T'$, we conclude by induction that $\tau'$ and $\upsilon'$ have
   the same lattice path sign if and only if the $E$-count of $\sigma'$ is even. Thus, we get
    \begin{equation}
        \lpsgn (\upsilon') =
        \begin{cases}
            (-1)^{k-s} \lpsgn (\tau') & \text{if $P$ is inside $\sigma$}, \\
            (-1)^{k} \lpsgn (\tau') & \text{if $P$ is outside $\sigma$}.
        \end{cases}
    \end{equation}

    \emph{Step $5$: Bringing it all together.}

    We consider the two cases separately:
    \begin{enumerate}
        \item Suppose $P$ is inside $\sigma$.  Then $c$ is odd. Hence,  the above considerations imply
            \begin{equation*}
                \begin{split}
                    \sgn (\lambda) & =  (-1)^{js}\sgn (\tilde{\lambda}) =  (-1)^{js + k-s}\sgn (\tilde{\mu}) \\
                                   & =  (-1)^{js + k-s + (j+c) s} \sgn ({\mu}) \\
                                   & = (-1)^k  \sgn ({\mu}),
                \end{split}
            \end{equation*}
            as desired.

        \item Suppose $P$ is outside of $\sigma$.  Then $c$ is even, and we conclude
            \begin{equation*}
                \begin{split}
                    \sgn (\lambda) & =  (-1)^{js}\sgn (\tilde{\lambda}) =  (-1)^{js + k}\sgn (\tilde{\mu}) \\
                                   & =  (-1)^{js + k-s + (j+c) s} \sgn ({\mu}) \\
                                   & = (-1)^k  \sgn ({\mu}).
                 \end{split}
            \end{equation*}
    \end{enumerate}

    Finally, it remains to consider the case where $P$ is the minimal covering region of two overlapping punctures of
    $T$. Let $\hat{T}$ be the triangular region that differs from $T$ only by having $P$ as a puncture, and let
    $\hat{\tau}$ and $\hat{\upsilon}$ be the tilings of $\hat{T}$ induced by $\tau$ and $\upsilon$, respectively. Since
    we order the end points of lattice paths using monomial labels, it is possible that the indices of the end points on
    the Northeast boundary of $P$ in $\tilde{T}$ differ from those of the points on the Northeast boundary of the
    overlapping punctures in $T$. However, the lattice paths induced by $\tau$ and $\upsilon$ connecting the points on
    the Northeast boundary of $P$ to the points on the Northeast boundary of the overlapping punctures are the same.
    Hence the lattice paths sign of $\tau$ and $\hat{\tau}$ differ in the same ways as the signs of $\upsilon$ and
    $\hat{\upsilon}$. Since we have shown our assertion for $\hat{\tau}$ and $\hat{\upsilon}$, it also follows for
    $\tau$ and $\upsilon$.
\end{proof}

Using difference cycles, we now characterise when two tilings of a region have the same lattice path sign.

\begin{corollary} \label{cor:lpsgn-cycle}
    Let $T = T_d(I)$ be a non-empty, balanced triangular region. Then two tilings of $T$ have the same lattice path sign
    if and only if the sum of the $E$-counts (which may count some end points $E_j$ multiple times) of the difference cycles
    is even.
\end{corollary}
\begin{proof}
    Suppose two tilings $\tau_1$ and $\tau_2$ of $T$ have difference cycles $\sigma_1, \ldots, \sigma_p$.
    By Lemma~\ref{lem:lpsgn-cycle}, $\lpsgn{\tau_1} = \lpsgn{\tau_2}$ if and only if an even number of the $\sigma_i$ have an odd $E$-count.
    The latter is equivalent to the sum of the $E$-counts of $\sigma_1, \ldots, \sigma_p$ being even.
\end{proof}

Our above results imply that the two signs that we assigned to a given lozenge tiling, the perfect matching sign (see
Definition~\ref{def:pm-sign}) and the lattice path sign (see Definition~\ref{def:nilp-sign}), are the same up to a
scaling factor depending only on $T$. The main result of this section follows now easily.

\begin{theorem} \label{thm:detZN}
    Let $T = T_d(I)$ be a balanced triangular region.  The following statements hold.
    \begin{enumerate}
        \item Let $\tau$ and $\tau'$ be two lozenge tilings of $T$. Then their perfect matching signs are
            the same if and only if their lattice path signs are the same, that is,
            \[
                \msgn (\tau) \cdot \lpsgn (\tau) = \msgn (\tau') \cdot \lpsgn (\tau').
            \]
        \item In particular, we have that
            \[
                |\det{Z(T)}| = |\det{N(T)}|.
            \]
    \end{enumerate}
\end{theorem}
\begin{proof}
    Consider two lozenge tilings of $T$. According to Corollaries~\ref{cor:msgn-cycle} and~\ref{cor:lpsgn-cycle}, they have
    the same perfect matching and the same lattice path signs if and only if the sum of the $E$-counts of the difference
    cycles is even. Hence using Theorems~\ref{thm:pm-matrix} and~\ref{thm:nilp-matrix}, it follows that $|\det{Z(T)}| =
    |\det{N(T)}|$.
\end{proof}

Theorem \ref{thm:detZN} allows us to move freely between the points of view using lozenge tilings, perfect matchings, and families
of non-intersecting lattice paths, as needed. In particular, it implies that rotating a triangular region by
$120^{\circ}$ or $240^{\circ}$ does not change the enumerations. Thus, for example, the three matrices described in
Remark~\ref{rem:rotations} as well as the matrix given in Example~\ref{exa:Z-matrix} all have the same determinant, up
to sign.

\subsection{A single sign}\label{sub:singlesign}~\par

We exhibit triangular regions such that all lozenge tilings have the same sign, that is, the signed and the unsigned enumerations are the same.  This is guaranteed to happen if all
floating punctures (see the definition preceding Lemma~\ref{lem:lpsgn-cycle}) have an even side length.

\begin{corollary} \label{cor:same-sign}
    Let $T$ be a tileable triangular region, and suppose all floating punctures of $T$ have an even side length. Then
    every lozenge tiling of $T$ has the same perfect matching sign as well as the same lattice path sign, and so
    $\per{Z(T)} = |\det{Z(T)}|$.

    In particular, simply-connected regions that are tileable have this property.
\end{corollary}
\begin{proof}
    The equality of the perfect matching signs follows from Corollary~\ref{cor:msgn-cycle}, and the equality of the lattice
    path signs from Corollary~\ref{cor:lpsgn-cycle}. Now Theorem~\ref{thm:pm-matrix} implies $\per{Z(T)} = |\det{Z(T)}|$.

    The second part is immediate as simply-connected regions have no floating punctures.
\end{proof}

\begin{remark}
    The above corollary vastly extends \cite[Theorem~1.2]{CGJL}, where hexagons are considered, using a different approach.
    This special case was also established independently in \cite[Section~3.4]{Ke}, with essentially the same proof as~\cite{CGJL}.

    Corollary~\ref{cor:same-sign} can also be derived from Kasteleyn's theorem on enumerating perfect
    matchings~\cite{Ka}. To see this, notice that in the case, where all floating punctures have even side lengths, all
    ``faces'' of the bipartite graph $G(T)$ have size congruent to $2 \pmod{4}$.
\end{remark}

We now extend Corollary~\ref{cor:same-sign}. To this end we define the \emph{shadow} of a puncture to be the region of
$T$ that is both below the puncture and to the right of the line extending from the upper-right edge of the puncture.
See Figure~\ref{fig:puncture-shadow}.

\begin{figure}[!ht]
    \includegraphics[scale=1]{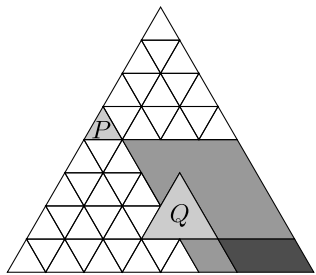}
    \caption{The puncture $P$ has the puncture $Q$ in its shadow (light grey), but $Q$ does not have a puncture in its shadow (dark grey).}
    \label{fig:puncture-shadow}
\end{figure}

\begin{corollary}\label{cor:same-sign-shadow}
    Let $T = T_d(I)$ be a balanced triangular region. If all floating punctures (and minimal covering regions of
    overlapping punctures) with other punctures in their shadows have even side length, then any two lozenge tilings of
    $T$ have the same perfect matching and the same lattice path sign. Thus, $\per{Z(T)} = |\det{Z(T)}|$.
\end{corollary}
\begin{proof}
    Let $P$ be a floating puncture or a minimal covering region with no punctures in its shadow. Then the shadow of $P$ is
    uniquely tileable, and thus the lozenges in the shadow are fixed in each lozenge tiling of $T$. Hence, no cycle of
    lozenges in any tiling of $T$ can contain $P$. Using Corollary~\ref{cor:msgn-cycle} and
    Corollary~\ref{cor:lpsgn-cycle}, we see that $P$ does not affect the sign of the tilings of $T$.

    Now our assumptions imply that all floating punctures (or minimal covering regions of overlapping punctures) of
    $T$ that can be contained in a difference cycle of two lozenge tilings of $T$ have even side length. Thus, we
    conclude $\per{Z(T)} = |\det{Z(T)}|$ as in the proof of Corollary~\ref{cor:same-sign}.
\end{proof}


\end{document}